\documentclass[12pt]{amsart}

\usepackage{amssymb,amsfonts}
\usepackage[all,arc]{xy}
\usepackage{enumerate}
\usepackage{mathrsfs}
\usepackage{graphicx}
\usepackage{caption}
\usepackage{subcaption}
\usepackage{subfig}
\usepackage{ textcomp }
\usepackage{ upgreek }
\usepackage{tikz}
\usepackage{leftidx}
\usepackage{enumitem}
\usepackage{nameref}
\usepackage[english]{babel}
\usetikzlibrary{calc}

\newtheorem{thm}{Theorem}[section]
\newtheorem{cor}[thm]{Corollary}
\newtheorem{prop}[thm]{Proposition}
\newtheorem{lem}[thm]{Lemma}
\newtheorem{conj}[thm]{Conjecture}

\theoremstyle{definition}
\newtheorem{defn}[thm]{Definition}

\def\<{\langle}
\def\>{\rangle}

\theoremstyle{remark}
\newtheorem{rem}[thm]{Remark}

\makeatletter
\let\c@equation\c@thm
\makeatother
\numberwithin{equation}{section}

\date{\today}

\begin{document}
\author[Chinyere]{I. Chinyere}
\address{ Ihechukwu Chinyere\\
Department of Mathematics and Maxwell Institute for Mathematical Sciences\\
Heriot--Watt University\\
Edinburgh EH14 4AS }
\email{ic66@hw.ac.uk}
\author[Howie]{J. Howie}
\address{ James Howie\\
Department of Mathematics and Maxwell Institute for Mathematical Sciences\\
Heriot--Watt University\\
Edinburgh EH14 4AS }
\email{J.Howie@hw.ac.uk}

   \title[Non-triviality of some one-relator products]{Non-triviality of some one-relator products of three groups}

   \begin{abstract}
    In this paper we study a group $G$ which is the quotient of a free product of three non-trivial groups 
  by the normal closure of a single element. In particular we show that if 
   the relator has length at most eight, then $G$ is non-trivial. In the case where the factors are cyclic, we
  prove the stronger result that at least one of the factors embeds in $G$.
   \end{abstract}
\keywords{One-relator products, pictures, representations.}
\subjclass{Primary 20F05, 20F06}
   \date{\today}
   \thanks{The first author was supported in this work
   by a Maxwell Scholarship from Heriot-Watt University.}
   \maketitle

\section{Introduction}
A one-relator product of groups is the quotient of a free product by the normal closure of a single element, called the \textit{relator}. In \cite{Fint} and \cite{Den2} (see also \cite{Howie8}) the following conjecture was proposed.

\begin{conj}\label{con}
A one-relator product on three non-trivial groups is non-trivial.
\end{conj}

Conjecture \ref{con} is an extension of the Scott-Wiegold conjecture (see Problem 5.53 in \cite{Kh}). The latter problem was solved by the second author \cite{Howie8}, who also conjectured that a free product of $(2n-1)$ groups is not the normal closure of $n$ elements.

\medskip
Our aim in this paper is to prove Conjecture \ref{con} under certain conditions. First we assume that the factors are all finite cyclic groups. Under this condition 
the conjecture is already known \cite{Howie8}, but
here we  prove a stronger result.

\begin{thm}\label{t1}
Let $G_a$, $G_b$ and $G_c$ be non-trivial cyclic groups with generators $a$, $b$ and $c$ respectively. For any word $w\in G_a*G_b*G_c$ whose exponent sum in each of the generators
is non-zero modulo the order of that generator, each of the factors $G_a$, $G_b$ and $G_c$ embed in
\[G=\dfrac{G_a*G_b*G_c}{N(w)}\]
via the natural maps.
\end{thm}

An immediate consequence (Corollary \ref{FHSfactor}) is that 
in any one-relator product $G$ of three cyclic groups, at least one
of the factors embeds in $G$; this is Conjecture 9.4 of \cite{Den2} for cyclic groups.

Also we can place a restriction on the length of the relator.
\begin{thm}\label{t2}
The one-relator product on three non-trivial groups is non-trivial when the relator has length at most eight.
\end{thm}

Placing  an upper bound on the relator allows us to apply techniques in combinatorial group theory such as pictures.

\medskip
The rest of the paper is organised as follows. Chapter 2 discusses pictures. As mentioned in the introduction, pictures combined with curvature arguments (discussed in Chapter 3) are used
in the proof of Theorem \ref{t2}. In Chapters 4 and 5, we give a proof of Theorems \ref{t1} and \ref{t2} respectively.

\medskip
Throughout, we shall use the following notations. Normal closure, $N(.)$; length, $\ell(.)$;  real part, $\Re(.)$; imaginary part, $\Im(.)$; homeomorphic, $\approx$; conjugate, $\simeq$; isomorphic, $\cong$; union $\cup$, and disjoint union, $\sqcup$.

\section{Pictures}
Pictures are one of the most powerful tools available in combinatorial group theory. Essentially, pictures are the duals of van Kampen diagrams \cite{Van1}. We will describe pictures briefly as it relates to groups with presentations of the form $G=\<X_1, X_2~|~R_1, R_2, R\>$, where $G_1=\<X_1~|~R_1\>$, $G_1=\<X_2~|~R_2\>$, and $R$ is a word with free product length at least two. 

\medskip
Groups of the form $G$ above are called \textit{one-relator products of groups $G_1$ and $G_2$}. In the next section we shall discuss such groups in more details. Pictures were first introduced by Rourke \cite{Rourke1} and adapted to work for such groups (as $G$) by  Short \cite{Short1}. Since then they have been used extensively and successfully by various authors in a variety of different ways (see [\cite{Duncan-1}, \cite{Duncan0}, \cite{Duncan2}, \cite{Gonzalez1}, \cite{Howie6}, \cite{Howie7}, \cite{How4}]). We describe below the basic idea, following closely the account in \cite{How3}. A more detailed description can be found in  \cite{Howie5} and also [\cite{Collins1}, \cite{Bogley1}, \cite{Hue}, \cite{Fenn1}, \cite{Pride2}].

\medskip
Let $G$ be as above, a picture $\Gamma$ over $G$ on an oriented surface $\mathcal{S}$ (usually $D^2$) consists of the following:

\begin{enumerate}
\item A collection of disjoint closed discs in the interior of $\mathcal{S}$ called vertices;

\item A finite number of disjoint arcs, each of which is either: 

\begin{enumerate}
\item  a simple closed curve in the interior of $\mathcal{S}$ that meets no vertex, 

\item   an arc joining two vertices (or one vertex to itself),
\item   an arc joining a vertex to the boundary $\partial \mathcal{S}$ of $\mathcal{S}$, or 

 \item  an arc joining $\partial \mathcal{S}$ to $\partial \mathcal{S}$; 
 \end{enumerate}
\item A collection of labels, one at each corner of each region of $\mathcal{S}$ (i.e. connected
component of the complement in $\mathcal{S}$ of the arcs and vertices) at a vertex, and one
along each component of the intersection of the region with $\partial \mathcal{S}$. The label at each corner is an element of $G_1$ or $G_2$.
Reading the labels round a vertex in the clockwise direction yields $R^{\pm 1}$
(up to cyclic permutation), as a cyclically reduced word in $G_1 * G_2$.
\end{enumerate}

\medskip
A region is a boundary region if it meets $\partial \mathcal{S}$, and an interior region otherwise.
If $\mathcal{S} \approx S^2$ or if $\mathcal{S} \approx D^2$ and no arcs of  meet $D^2$, then $\Gamma$ is called \textit{spherical}. In
the latter case $\partial D^2$ is one of the boundary components of a non-simply connected
region (provided, of course, that $\Gamma$ contains at least one vertex or arc), which is
called the \textit{exceptional region}. All other regions are interior.
The labels of any region $\triangle$ of $\Gamma$ are required all to belong to either $G_1$ or $G_2$. Hence we
 can refer to regions as $G_1$-regions and $G_2$-regions accordingly. Similarly a corner is called a $G_i$-corner or more specially a $g_i$-corner if it is labelled by the element $g_i\in G_i$. Each arc is required to separate a
$G_1$-region from a $G_2$-region. Observe that this is compatible with the alignment of
regions around a vertex, where the labels spell a cyclically reduced word, so must
come alternately from $G_1$ and $G_2$. A region bounded by arcs that are closed curves
will have no labels; nevertheless the above convention requires that it be designated
a $G_1$- or $G_2$-region. An important rule for pictures is that the labels within any
$G_1$-region (respectively $G_2$-region) allow the solution of a quadratic equation in $G_1$
(respectively $G_2$). The labels around any given boundary component of the region
are formed into a single word read anti-clockwise. The resulting collection of
elements of $G_1$ or $G_2$ is required to have genus no greater than that of the region
(in the sense of \cite{Duncan2}). This technical general requirement is much simpler in the
commonest case of a simply connected region - it means merely that the resulting
word represents the identity element in $G_1$ or $G_2$.
 
\medskip
 Two distinct vertices of a picture are said to  \emph{cancel} along an arc $e$ if they are joined by $e$ and if their labels, read from the endpoints of $e$, are mutually inverse
words in $G_{1} * G_{2}$. Such vertices can be removed from a picture via a sequence of \textit{bridge moves} (see Figure \ref{bridge} and  \cite{Duncan2} for more details), followed by deletion of a \textit{dipole} without changing the boundary label. A \textit{dipole} is a connected spherical picture containing precisely two vertices, does not meet $\partial \mathcal{S}$, and none of its interior regions contain other components of $\Gamma$. This gives an alternative picture with the same boundary label and two fewer vertices. 

\begin{figure}[h!]
\centering
\includegraphics[scale=0.20]{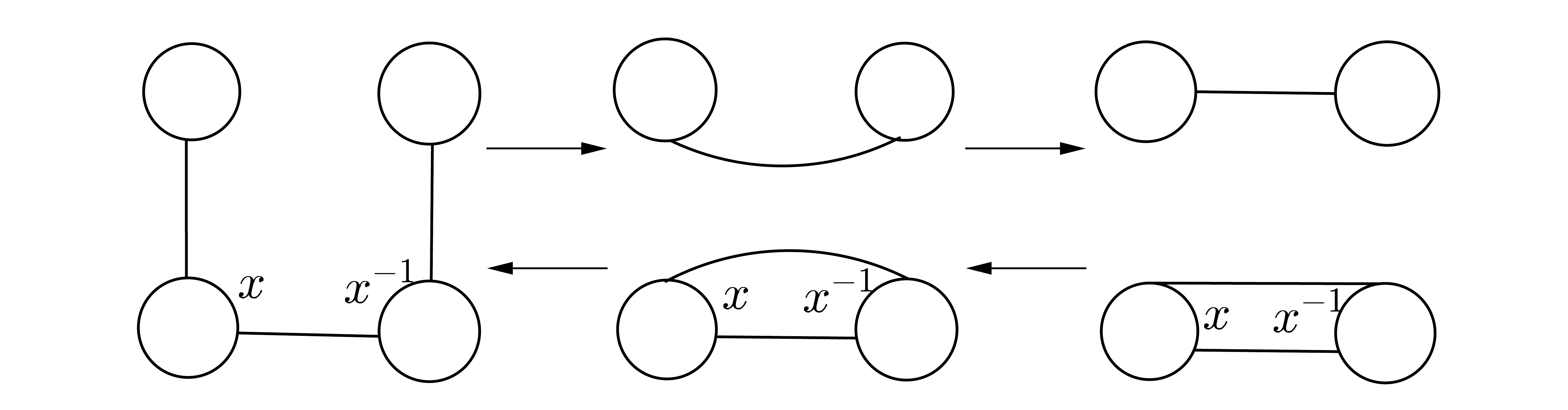}
\caption{Diagram showing bridge-move.}
\label{bridge}
\end{figure}
\medskip
We say that a picture $\Gamma$ is \textit{reduced} if it cannot be altered by bridge moves to a picture with a pair of cancelling vertices. If $\mathcal{W}$ is a set of words, then a picture is $\mathcal{W}$-\textit{minimal} if it is non-empty and has the minimum number of vertices  amongst all pictures over $G$ with boundary label in $\mathcal{W}$. Any cyclically reduced word in $G_{1} * G_{2}$ representing the identity element of $G$ occurs as the boundary label of some reduced picture on $D^2$. A picture is \textit{connected} if the union of its vertices and arcs is connected. In particular, no arc of a connected picture is a closed arc or joins two points of $\partial \mathcal{S}$, unless the picture consists only of that arc.  

\medskip
Two arcs of $\Gamma$ are said to be \textit{parallel} if they are the only two arcs in the boundary of some simply-connected region $\bigtriangleup$ of $\Gamma$. We will also use the term \textit{parallel} to denote the equivalence relation generated by this relation, and refer to any of the corresponding equivalence classes as a \textit{class of $\omega$ parallel arcs} or \textit{$\omega$-zone}. Given a \textit{$\omega$-zone} joining vertices $u$ and $v$ of $\Gamma$, consider the $\omega- 1$ two-sided regions separating these arcs. Each such region has a corner label $x_{u}$ at $u$ and a corner label $x_{v}$ at $v$, and the picture axioms imply that $x_{u}x_{v} = 1$ in $G_{1}$ or $G_{2}$. The $\omega -1$ corner labels at $v$ spell a cyclic subword $s$ of length $\omega-1$ of the label of $v$. Similarly the corner labels at $u$ spell out a cyclic subword $t$ of length $\omega -1$. Moreover, $s=t^{-1}$. If we assume that $\Gamma$ is reduced, then $u$ and $v$ do not cancel. Hence the cyclic permutations of the labels at $v$ and $u$ of which $s$ and $t$ are initial segments respectively are not equal. Hence $t$ and $s$ are \textit{pieces}. 

\section{Combinatorial curvature}\label{subsec}
 For any compact orientable surface (with or without boundary) $\mathcal{S}$ with a triangulation, we assign real
numbers $\beta$ to the corners of the faces in $\mathcal{S}$. We will think of these numbers as interior angles. A vertex which is on $\partial \mathcal{S}$
is called a boundary vertex, and otherwise interior. The \textit{curvature} of an interior vertex $v$ in $\mathcal{S}$ is defined as

\begin{equation*}
\kappa(v) = \left[ 2 - \sum_i {\beta(v)}_i\right] \pi,
\end{equation*}

where the ${\beta(v)}_i$ range over the angles at $v$. If $v$ is a boundary vertex then we define

\begin{equation*}
\kappa(v) = \left[ 1 -\sum_i {\beta(v)}_i\right] \pi.
\end{equation*}

The curvature of a face $\Delta$ is defined as 

\begin{equation*}
\kappa(\Delta) = \left[ 2-d(\Delta)+\sum_i {\beta(\Delta)}_i\right] \pi
\end{equation*}

where ${\beta(\Delta)}_i$ are the interior angles
of $\Delta$. The combinatorial version of the Gauss-Bonnet theorem states that the total curvature is the multiple of 
Euler characteristic of the surface by $2\pi$:

\begin{equation*}
\kappa(\mathcal{S}) = \left[ \sum_v {\kappa(v)}+\sum_\Delta {\kappa(\Delta)}\right] \pi=2\pi\chi(\mathcal{S}).
\end{equation*}

We use curvature to prove results by showing that this value cannot be realised. We assign to each corner of a region of degree $k$ an angle $(k-2)/k$. This will mean that regions are flat in the sense that they have zero curvature (alternatively we can make vertices flat instead). This will be the standard assignment for this work. In other words, wherever curvature is mentioned with no specified assignments, it is implicitly assumed that we are using the one described above.

\medskip
In some cases, it may be needful to redistribute curvature (see \cite{Edjvet2}). This involves locating positively-curved vertices (or regions), and using its excess curvature to compensate its negatively curved neighbours. Hence the total curvature is preserved. We shall describe how to do this in Section 5.
\section{One-relator product of cyclics}
In this Section we give a proof of Theorem \ref{t1}. Recall that $G_a=\<a~|~a^p\>$, $G_b=\<b~|~b^q\>$, $G_c=\<c~|~c^r\>$, and $w$ is a word in the free product $G_a*G_b*G_c$ with non-zero exponent sum in each of the generators $a$, $b$ and $c$. The proof we present follows closely the one in \cite{Howie8}.
\begin{lem}\label{Fr}
Suppose that $p$, $q$ and $r$ are  
prime powers. Then each of $G_{a},G_{b},G_{c}$ embeds via the natural map into $$G=\frac{(G_{a}*G_{b}*G_{c})}{N(w)}.$$
\end{lem}

\begin{proof}
We know the result holds when $p,q,r$ are primes by [\cite{Howie8}, Theorem 4.1]. He we assume  that at least one of $p,q,r$ is a prime power but not prime.

\medskip
Suppose that $n$ is the exponent sum of $a$ in $w$. The assumption that $n$ is not divisible by $p$ implies that $n=tp+s$ with $0<s<p$. By replacing $w$ with $wa^{-tp}$ which changes $n$ to $s$ (leaving $G$ unchanged), we can always assume that $n<p$.
If $m$ is co-prime to $p$ then $a\mapsto a^m$ induces an automorphism of
$G_{a}$.  Thus, replacing $a$ by $a^m$ in $w$ gives a new word $w'\in G_{a}*G_{b}*G_{c}$
such that the resulting group
$$G'=\frac{(G_{a}*G_{b}*G_{c})}{N(w')}$$
is isomorphic to $G$ (and such that $G_{a}$ embeds in $G'$ if and only if it embeds in $G$).  Moreover, the exponent sum of $a$ in $w'$ is $mn$.  By Bezout's Lemma
we may choose $m$ such that $mn\equiv gcd(n,p)~\mathrm{mod}~p$.
Thus without loss of generality we may assume that $n$ divides $p$ 
(and similarly the exponent sums of $b,c$ in $w$ divide $q,r$ respectively). So in particular if $p$ is prime, then $n=1$.

\medskip
Now suppose that $p$ and $n$ (the exponent sum of $a$ in $w$) are
powers of a prime $\tau$ -- say $p=\tau^t$ and $n=\tau^s$ where $0\le s<t$.
If $\tau$ is an odd prime, define
\begin{equation*}
\theta_p=\frac{(\tau^{t-s}-1)\pi}{2\tau^t}.
\end{equation*}
If $\tau=2$, define

\begin{equation*}
\theta_p=\frac{(2^{t-s-1}-1)\pi}{2^t}
\end{equation*}

unless $s=t-1$, in which case define $\theta_p=\frac{\pi}2$. Recall that an element $\cos(\theta_p) + v\sin(\theta_p)\textbf{v}\in SO(3)\approx S^3/\lbrace \pm 1\rbrace$ has order $p$ if and only if $\theta_p$ is a multiple of $\pi/p$ but not $\pi/{\tau^{t-1}}$, for any vector $\textbf{v}\in S^2$.
Hence for any  $\textbf{v}\in S^2$, the map
$$\alpha_\textbf{v}:G_{a}\to\mathbb{H},~ a\mapsto \cos(\theta_p) + \sin(\theta_p)\textbf{v},$$
induces a faithful representation $G_{a}\to SO(3)$ (where
$\mathbb{H}$ denotes the quaternions).  Moreover, $\Re(\alpha_v(a^n))=\cos(\psi_p)$ where $\psi_p:=n\theta_p$ and $\frac\pi4\le \psi_p\le\frac\pi2$.

\medskip
Similarly, we can define maps $\beta_\textbf{v}:G_{b}\to\mathbb{H}$ and
$\gamma_\textbf{v}:G_{c}\to\mathbb{H}$ that induce faithful representations 
$G_{b},G_{c}\to SO(3)$, and such that, if $e,f$ denote the exponent-sums of
$b,c$ in $w$, then $\Re(\beta_\textbf{v}(b^e))=\cos(\psi_q)$ and $\Re(\gamma_\textbf{v}(c^f))=\cos(\psi_r)$ where $\psi_q,\psi_r\in [\frac\pi4,\frac\pi2]$.

\medskip
The numbers $\psi_p$, $\psi_q$, and $\psi_r$ satisfy a triangle inequality. In other words none
is greater than the sum of the other two.  Hence, for example,
the triple $(\alpha_{\textbf{i}},\beta_{\textbf{i}},\gamma_{-\textbf{i}})$ induces a homomorphism

\begin{equation*}
\delta:G_{a}*G_{b}*G_{c}\to S^3
\end{equation*}

that sends $w$ to $\cos(\theta)+\textbf{i}\sin(\theta)$ with
$0\le\theta\le\frac{3\pi}4$.  If $\theta=0$ then $\delta$ induces a
representation $G\to SO(3)$ that is faithful on each of $G_{a},G_{b},G_{c}$ and we are done.
So assume that $\theta>0$.  In other words $\Im(\delta(w))>0$.

\medskip
Similar remarks apply to the triples $(\alpha_\textbf{i},\beta_{-\textbf{i}},\gamma_\textbf{i})$
and
$(\alpha_{-\textbf{i}},\beta_{\textbf{i}},\gamma_\textbf{i})$.  Hence also the triple
$(\alpha_\textbf{i},\beta_{-\textbf{i}},\gamma_{-\textbf{i}})$ induces a representation $\delta$ with $\Im(\delta(w))<0$.
The map $S^2\to S^3$, $\textbf{v}\mapsto (\alpha_\textbf{i},\beta_{-\textbf{i}},\gamma_\textbf{v})(w)$ is an $S^1 $-equivariant map under the conjugation action. It follows that it
either sends some $\textbf{v}$ to $\pm1\in S^3$ (in which case $(\alpha_\textbf{i},\beta_{-\textbf{i}},\gamma_\textbf{v})$ gives a representation $G\to SO(3)$ that is faithful on each of
$G_{a},G_{b},G_{c}$), or by [\cite{Howie8} Corollary 2.2], it represents $+1\in H_2(S^3-\lbrace \pm 1\rbrace)\cong \mathbb{Z}$.

\medskip
Similarly, the map $\textbf{v}\mapsto (\alpha_\textbf{i},\beta_\textbf{i},\gamma_\textbf{v})(w)$ either maps some
$\textbf{v}\in S^2$ to $\pm 1\in S^3$ and so gives a representation $G\to SO(3)$
that is faithful on each of $G_{a},G_{b},G_{c}$, or represents one of $0,-1\in H_2(S^3-\lbrace\pm 1\rbrace)\cong\mathbb{Z}$.  Now any path $P:[0,1]\to S^2$
from $-\textbf{i}$ to $\textbf{i}$ gives rise to a homotopy

\begin{equation*}
t\mapsto \left(\textbf{v}\mapsto (\alpha_i,\beta_{P(t)},\gamma_v)(w)\right)
\end{equation*}

between the above two maps.  If $(\alpha_\textbf{i},\beta_{P(t)},\gamma_\textbf{v})(w)\ne\pm 1$
for all $t$ and for all $\textbf{v}$, then we can regard this as a homotopy of maps
$S^2\to S^3\setminus\{\pm 1\}$. This is a contradiction since the two maps belong
to different homology classes in $H_2(S^3\setminus\{\pm 1\})$.
Hence for some $t$ and some $\textbf{v}\in S^2$, the map 
$(\alpha_\textbf{i},\beta_{P(t)},\gamma_\textbf{v})$
sends $w$ to $\pm 1$ and so
induces a representation $G\to SO(3)$ that is faithful on each of $G_a, G_b$ and $G_c$.

It follows that each of the natural maps from $G_a, G_b$ and $G_c$ to $G$ is injective, as
required. 
\end{proof}

The general case of the Freiheitssatz for $G$ follows from the special case of prime powers together with
the Chinese Remainder Theorem by an easy induction.

\begin{proof}[Proof of Theorem \ref{t1}]
For the inductive step, suppose that $p=mn$ with $gcd(m,n)=1$.  Since the exponent sum of the generator
$a$ in $w$ is non-zero modulo $p$, we can assume that it is non-zero
modulo $m$.  Now factor out $a^m$ and apply the inductive hypothesis.  This shows
that the maps $G_{b}\to \mathbb{H}$ and $G_{c}\to \mathbb{H}$ are injective.  It also shows that
the  kernel $K$ of $G_{a}\to \mathbb{H}$ is contained in the subgroup $\<m\>$.

\medskip
If the exponent sum of $a$ in $w$ is also non-zero modulo $n$, then by interchanging the roles of
$m$ and $n$ in the above we see that $K$ is contained in $\<n\>$.  However, if the exponent-sum of
$a$ in $w$ is divisible by $n$, then the same is  automatically true: $K$ is contained in $\<n\>$.
Finally, we know that $K$ is contained in the intersection of  $\<m\>$ and $\<n\>$.  But this intersection is trivial
by the Chinese Remainder Theorem, so we deduce that $G_{a}\to \mathbb{H}$ is injective.
\end{proof}

In a one-relator product of groups $G=(\ast_\lambda G_\lambda)/N(R)$, we say that a factor group $G_\lambda$ is a {\em Freiheitssatz factor} if the natural map $G_\lambda\to G$ is injective.  It is clear that any $G_\lambda$ such that the product of the $G_\lambda$-letters in $R$ is trivial is a Freiheitssatz factor.  Combining this remark with Theorem \ref{Fr} we obtain:

\begin{cor}\label{FHSfactor}
Any one-relator product of three cyclic groups contains a Freiheitssatz factor.
\end{cor}

In \cite{Chiodo1}, Chiodo used the result of \cite{Howie8} to show that the free product $G$ of three cyclic groups of distinct prime orders is \textit{finitely annihiliated}. In other words, for every non-trivial element $g\in G$, there exist a finite index normal subgroup $N$ of $G$ such that $g$ is trivial in $G/N$. The proof uses nothing more than the fact that finitely generated subgroups of $SO(3)$ are residually finite. Hence our result extends this to the case where the cyclic groups are arbitrary.

\begin{cor}
Any free product of three cyclic groups is finitely annihiliated.
\end{cor}
\section{One-relator product with short relator}
In this Section we give a proof of Theorem \ref{t2}. As mentioned in the introduction the proof uses pictures, as well as Bass-Serre theory and Nielsen transformations, and is broken down into a number of Lemmas. Also we shall need the following results.

\begin{thm}\label{t33}
Suppose that $A$ and $B$ are  non-cyclic
two-generator groups with generators $\lbrace a, c\rbrace$ and $\lbrace b, d\rbrace$ respectively. If $A$ and $B$ have faithful representations in $PSL_2(\mathbb{C})$, then $G=(A*B)/N(abcd)$ satisfies the Freiheitssatz: the natural maps $A\to G$ and $B\to G$ are injective.
\end{thm}
\begin{proof}
Let $X,Y$ and $Z$ be variable matrices in $SL_2(\mathbb{C})$. The aim of the proof is to show that one can choose $X$, $Y$ and $Z$ such that $a\mapsto X$, $c\mapsto Z$ gives  a faithful  representation $A\mapsto  PSL_2(\mathbb{C})$ and $b \mapsto Y$, $d \mapsto (XYZ)^{-1}$ gives  a faithful  representation $B \mapsto PSL_2(\mathbb{C})$.

Such a triple of matrices is a representation
of the free group $F_3$ of rank 3.  Recall \cite{Goldman1}
that the character variety of representations $F_3\to PSL(2,\mathbb{C})$ is given by the seven parameters
Tr$(X$), Tr$(Z)$, Tr$(Y)$, Tr$(XY)$, Tr$(XZ)$, Tr$(YZ)$ and Tr$(XYZ)$ subject to a single polynomial equation

\begin{align}\label{e1}
\begin{split}
&\mathrm{Tr}(X)^2+\mathrm{Tr}(Y)^2+\mathrm{Tr}(Z)^2+
\mathrm{Tr}(XY)^2+\mathrm{Tr}(XZ)^2+\mathrm{Tr}(YZ)^2\\
&+
\mathrm{Tr}(XYZ)^2+\mathrm{Tr}(XY)\mathrm{Tr}(XZ)\mathrm{Tr}(YZ)\\
&
-\mathrm{Tr}(X)\mathrm{Tr}(Y)\mathrm{Tr}(XY)
-\mathrm{Tr}(X)\mathrm{Tr}(Z)\mathrm{Tr}(XZ)-\mathrm{Tr}(Y)\mathrm{Tr}(Z)\mathrm{Tr}(YZ)\\
&
+\mathrm{Tr}(X)\mathrm{Tr}(Y)\mathrm{Tr}(Z)\mathrm{Tr}(XYZ)
-\mathrm{Tr}(X)\mathrm{Tr}(YZ)\mathrm{Tr}(XYZ)\\
&-\mathrm{Tr}(Y)\mathrm{Tr}(XZ)\mathrm{Tr}(XYZ)
-\mathrm{Tr}(Z)\mathrm{Tr}(XY)\mathrm{Tr}(XYZ)
=4
\end{split}
\end{align}

\medskip
\noindent
By hypothesis, faithful  representations of $A$ and $B$ in $PSL_2(\mathbb{C})$ exist. Moreover they are parametrised by fixing suitable values for Tr$(X$), Tr$(Z)$, Tr$(XZ)$, Tr$(Y)$, Tr$(XYZ)$ and $\text{Tr}(XYZY^{-1})$. By the repeated application of trace relation 
\begin{equation*}
\text{Tr}(MN)=\text{Tr}(M)\text{Tr}(N)-\text{Tr}(MN^{-1})
\end{equation*}
for arbitrary matrices $M$ and $N$ can write:
\begin{equation*}
\text{Tr}(XYZY^{-1})\!=\! \text{Tr}(Y)\text{Tr}(XYZ)-\text{Tr}(XY) \text{Tr}(YZ)+\text{Tr}(X) \text{Tr}(Z) -\text{Tr}(XZ).
\end{equation*}
Hence if we fix suitable values for Tr$(X)$, Tr$(Y)$, Tr$(Z)$, Tr$(XZ)$ and Tr$(XYZ)$, we  have two free variables  $\alpha:=\text{Tr}(XY)$ and $\beta:=\text{Tr}(YZ)$ which are required to satisfy the quadratic equation which fixes the value of 

\begin{equation*}
\text{Tr}(XYZY^{-1})= \text{Tr}(Y)\text{Tr}(XYZ)-\alpha \beta+\text{Tr}(X) \text{Tr}(Z) -\text{Tr}(XZ).
\end{equation*}

Combining this with Equation (\ref{e1}), and fixing
Tr$(X$), Tr$(Z)$, Tr$(XZ)$, Tr$(Y)$, Tr$(XYZ)$,
we have a pair of quadratic equations in $\alpha,\beta$ of the form

$$
\alpha\beta=c_1,$$
$$\alpha^2+\beta^2+c_2\alpha\beta+c_3\alpha+c_4\beta=c_5
$$
for suitable constants $c_1,\dots,c_5$.  It is routine to check that any such pair of equations can be solved in $\mathbb{C}$.  Any solution
 gives a representation  $\langle a,b,c,d \rangle \longrightarrow
SL_2(\mathbb{C})$ that induces the given faithful representations of $A$ and $B$ in  $PSL_2(\mathbb{C})$
 up to conjugacy, mapping the word $abcd$ to the identity element. This completes the proof.
\end{proof}

\begin{thm}\label{t3}
Let $G$ be a one-relator product of non-trivial groups $A$ and $B$, with relator $r^n$ for some integer $n$. If $2\leq\ell(r)\leq 6$ and $n\geq 2$, then  $r$ has order $n$ in $G$.
\end{thm}

Theorem \ref{t3} is a consequence of various results proved in Chapter 4 of \cite{ihe}. We omit the proof
which is straightforward but lengthy. It uses standard 
curvature arguments on pictures.

\medskip
Theorem \ref{t1} holds trivially if $w$ is in the normal closure of any of the factors. Hence we can assume by Theorem \ref{t3} that $w$ contains at least two letters in each of the three factors. To see this, suppose that $w$ contains  one letter from $A$, say $\alpha$. Then a cyclic conjugate of $w$ has the form $\alpha W$, where $W\in B*C$ and $\ell(W)\leq 7$
(and so some conjugate of $W$ has length at most $6$). If $n$ is the order of $\alpha$ in $A$, then by Theorem \ref{t3}, $G$ is the free product of $A$ and $(B*C)/N(W^n)$ amalgamated over the subgroups $\<\alpha\>$ and $\<W\>$ of $A$ and $B*C$ respectively. So $G$ is non-trivial and in particular we can assume that $\ell(w)\geq 6$.

\medskip
We can also assume that up to cyclic permutation $w$ has the form $c_1Uc_2V$, where $c_1,c_2\in C$ and $U,V\in A*B$, with $\ell(U)+\ell(V)\leq 6$. The group $G$ is non-trivial if $c_1c_2=1$ or $UV\in N(A)\cup N(B)$. Hence we assume that neither of the two conditions holds. If without loss of generality we assume $\ell(U)\leq \ell(V)$, then the possibilities for $U$ and $V$ as words in the free product $A*B$ are as follows:
\medskip
\begin{enumerate}\label{Niel}
\item $U$ or $V$ is in the normal closure of $A$ or $B$;
\item $U=(\alpha\beta)^{\pm 1}$ and $V \in \lbrace \alpha_1 \beta_1, \alpha_1\beta_1\alpha_2, \alpha_1\beta_1\alpha_2\beta_2 \rbrace;$
\item $U=\alpha\beta \alpha_1$ and $V \in \lbrace \beta_1\alpha_2 \beta_2, \alpha_2\beta_1\alpha_3\rbrace$ (by symmetry),
\end{enumerate}
where $\alpha$ (with or without subscript) is an element of $A$ and $\beta$ (with or without subscript) is an element of $B$. We show that in each of the possibilities listed above $G$ is non-trivial.

\begin{lem}\label{news1}
If $U$ or $V$ is in the normal closure of $A$ or $B$, then $G$ is non-trivial.
\end{lem}
\begin{proof}
Without loss of generality, we can assume that $U$ is in the normal closure of $A$. In other words, there exist a word $\gamma\in A*B$ such that $U=\gamma^{-1}\alpha\gamma$.
Hence 
\begin{align*}
w=& c_1 \gamma^{-1}\alpha\gamma c_2 V \\
\simeq &\gamma c_1 \gamma^{-1}\alpha\gamma c_2\gamma^{-1} \gamma V \gamma^{-1}.
\end{align*}
So $G$ is trivial if and only if 
\begin{equation*}
G'=\dfrac{(A*B*\tilde{C})}{N(W)}
\end{equation*}
is trivial, where $\tilde{C}=\gamma C\gamma^{-1}$ and $W=\tilde{c_1}\alpha\tilde{c_2}V$. Hence it is enough to consider the case where $U=\alpha$ (i.e $\gamma=1$).

\medskip
By assumption $c_1\ne c_2^{-1}$. Hence $c_{1}\alpha c_{2}$ has infinite order in $A*C$. It follows that the subgroup of $A*C$ generated by $A$ and $c_{1}\alpha c_{2}$ is isomorphic to the free product $A*\mathbb{Z}$. If $A$ and $V$ generate a subgroup of $A*B$ which is also isomorphic to $A*\mathbb{Z}$, then 
\begin{equation*}
G=(A*B)\leftidx{_{{\langle A,V \rangle}}}{*}_{{\langle A,c_{1}\alpha c_{2} \rangle}}(A*C).
\end{equation*}
So in this case $G$ is non-trivial.

\medskip
Suppose then that $\langle A,V \rangle$ is not isomorphic to $A*\mathbb{Z}$. We can assume that $V$ contains at least two $B$ letters and $\ell(V)\leq 5$. If $V$ contains exactly two $B$ letters, then the two letters must be inverses  of each other. This implies that $B$ is a homomorphic image of $G$, so $G$
is non-trivial. Hence $V$ contains exactly three $B$ letters. It follows that $V$ is conjugate in $A*B$ to a letter $\beta\in B$ with order $r<\infty$. Define $H$ to be the group
\begin{align*}
H:=&\dfrac{A*C}{N((c_{1}\alpha c_{2})^r)}\\
=&A ~ {_{\langle \alpha\rangle}^{~*}}~ T~ {_{\langle \gamma\rangle =\langle c_{2}c_{1}\rangle}^{~~~~*}}~ C,
\end{align*}
where $T=\langle \alpha, \gamma ~|~  \alpha^p, \gamma^q, (\alpha\gamma)^r\rangle$ and $p,q$ are the orders of $\alpha$ and $c_{2}c_{1}$ respectively. Then $G=(A*B)\leftidx{_{{\langle A,V \rangle}}}{*}_{{\langle A,c_{1}\alpha c_{2} \rangle}}H,$ provided of course that $\langle A,c_1\alpha c_2\rangle$ embeds in $H$. We show below that this is in fact the case.

\medskip
By Bass-Serre theory, $H$ acts on a tree $\Gamma$. Let the vertex set of $\Gamma$ be $X$. The edge $e$ divides $\Gamma$ (see Figure \ref{TReee}) into two components $\Gamma_1$ and $\Gamma_2$ with vertex sets $X_1$ (containing vertices $u_A$ and $u_T$) and $X_2$ (containing vertices $u_C$ and $u_{c_1{(u_T)}}$) respectively such that $X=X_1 \sqcup X_2$. The vertices $u_A$, $u_T$ and $u_C$ have stabilizers $A$, $T$ and $C$ respectively. Similarly,  $\langle \alpha\rangle$ and  $\langle \gamma\rangle$ are the stabilizers of $e_1$ and $e$ respectively. Let $e_2=c_{1}(e)$, then the stabilizer of $e_2$ is $c_{1}\langle \gamma\rangle c_{1}^{-1}$.

\begin{figure}[h!]
\centering
\includegraphics[scale=0.23]{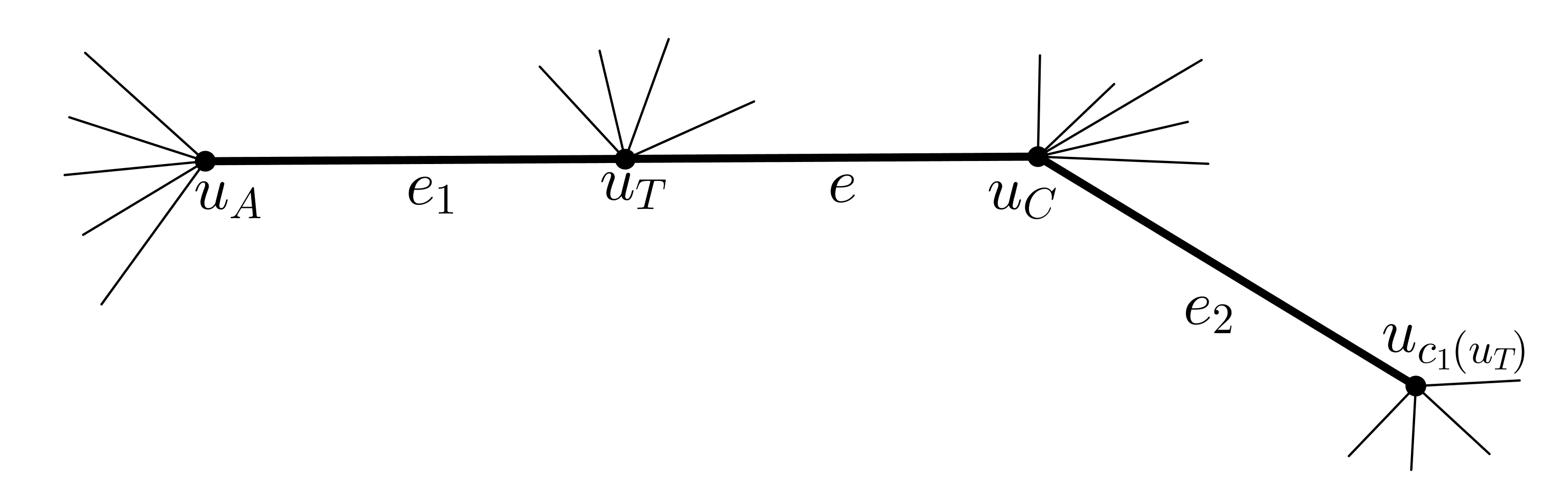}
\caption{Diagram showing a section of the tree $\Gamma$ on which $H$ acts.}
\label{TReee}
\end{figure}

We aim to apply the Ping-Pong Lemma to show
that the subgroup of $H$ generated by $A$ and $\langle c_1\alpha c_2\rangle$ is their free product.  Since 
$$A\cap  \langle \gamma\rangle=\langle\alpha\rangle\cap  \langle \gamma\rangle=1$$
it follows that $a(X_2)\subset X_1$ for all $a\ne 1$ in $A$.

In a similar way, since $\alpha\gamma$ stabilizes an edge $e_3$ incident to $u_T$ other than $e$, $c_{1}\alpha c_{2}=c_{1}\alpha\gamma c_{1}^{-1}$ stabilizes an edge 
$e_4=c_1(e_3)$ incident to $u_{c_1(u_T)}$ other than $c_1(e)=e_2$.  If we can show that
$\langle \gamma\rangle\cap  \langle c_{1}\alpha c_{2}\rangle=1$,
then it follows that $b(X_1)\subset X_2$ for every $b\ne 1$
in $\langle c_{1}\alpha c_{2}\rangle$, and the Ping-Pong Lemma
will yield the result.

\medskip
But $\langle \gamma\rangle\cap  \langle c_{1}\alpha c_{2}\rangle$ stabilizes $e$ and $e_4$, and hence also $e_2$, so it is contained in $c_{1}\langle \gamma\rangle c_{1}^{-1}$.
Hence, $\langle \gamma\rangle\cap \langle \alpha\gamma\rangle = 1$ in $T$ implies that in $c_{1}Tc_{1}^{-1}$,
\begin{equation*}
c_{1}\langle \gamma\rangle c_{1}^{-1}\cap c_{1}\langle \alpha \gamma\rangle c_{1}^{-1}= 1.
\end{equation*}

Thus $\langle \gamma\rangle\cap  \langle c_{1}\alpha c_{2}\rangle=1$, as required.
\end{proof}

It follows in particular from Lemma \ref{news1} that $\ell(U)\geq 2$. The rest of the arguments we present rely heavily on Nielsen transformations. We transform $\lbrace U, V \rbrace$ into a more suitable  Nielsen equivalent set depending on the subgroup they generate.
\begin{lem}\label{news2}
Suppose $\langle U,V\rangle$ is free. Then $G$ is non-trivial.
\end{lem}
\begin{proof}
First we suppose $\langle U,V\rangle$ is free of rank $1$ say with generator $t$. Then $w$ can be expressed in the form $w=c_1 t^r c_2 t^s$, where $V=t^s$ and $U=t^r$ for integers $s,r$. If $s+r=0$, then $G\neq 1$ since $w\in N(C)$. Otherwise $w=1$ is a non-singular equation over $C$. We assume that $t\not\in N(A)\cup N(B)$ for otherwise $G\neq 1$ by Lemma \ref{news1}. 
It follows that any cyclically reduced conjugate of $t$ has length at least $2$. So since $s,r\neq 0$ and $\ell(t^n)\geq 2n$ for any $n$,

\begin{equation*}
 |s|+|r|\leq 3.
\end{equation*}
Consider the group 
$$H:=\dfrac{(C*\langle t\rangle)}{N(w)}.$$
If $s,r\geq 1$, then $C$ embeds in $H$ by \cite{Levin1}. Otherwise without loss of generality $s=-1$ and $r=2$. Again $C$ embeds in $H$ by \cite{Howie2}. If $t$ is trivial in $H$,  then $H=C$ and so $w\in N(A*B)$; again $G\neq 1$. If $t$ has finite order $m>1$ in $H$, then 
$$G=\dfrac{(A*B)}{N(t^m)}*_{{\langle t\rangle}}H.$$ 
By the previous comment it follows that $\ell(t)=2$, say $t=\alpha\beta$. Hence  
$$(A*B)/N(t^m)\cong A \ast_{\langle\alpha\rangle} T \ast_{\langle\beta\rangle} B,$$
where $T=\langle\alpha,\beta|\alpha^{|\alpha|}=\beta^{|\beta|}=(\alpha\beta)^m=1\rangle$ is a triangle group, and so $t$ has order $m$ as required.

\medskip
Finally if $t$ has infinite order in $H$, then 
$$G=(A*B)*_{{\langle t\rangle}}H.$$ Hence $G$ is non-trivial.

\medskip
Now suppose $\langle U,V\rangle$ is free of rank $2$. Let $$H:=(C*\langle U,V\rangle)/N(w)=C*\langle U\rangle.$$ Note that the subgroup of $H$ generated by $\lbrace U,c_1Uc_2\rbrace$ is a free group of rank $2$. It follows that $G$ is the free product of $H$ and $A*B$ amalgamated over the subgroups $\langle U,c_1Uc_2\rangle$ and $\langle U,V\rangle$. It follows that $G$ is non-trivial.
\end{proof}

\begin{lem}\label{news3}
Suppose $\langle U,V\rangle$ is isomorphic to  $C_p*C_q$  or $C_p*\mathbb{Z}$, where $C_p$ and $C_q$ are finite cyclic groups. Suppose further that $\ell(U)\le \ell(V)<4$. Then $G$ is non-trivial.
\end{lem}
By assumption $\langle U,V\rangle$ has an element of finite order. So Nielsen transformations can be applied to $\lbrace U,V\rbrace$ to get a new set $\lbrace u,v\rbrace$ with $u$ or $v$ having finite order. Note also that $V\neq \alpha_1\beta_1\alpha_2\beta_2$. Lemma \ref{news3} is a corollary to Propositions \ref{pp1}--\ref{pp4} below.

\begin{prop}\label{pp1}
If $U=( \alpha\beta)^{\pm 1}$ and $V= \alpha_1\beta_1$, then $G$ is non-trivial.
\end{prop}

\begin{proof}
First suppose $U=( \alpha\beta)^{-1}$. Since $UV$ is not in the normal closure of $A$ or $B$, neither $\alpha=\alpha_1$ nor  $\beta=\beta_1$ holds. Hence $\langle U,V\rangle$ is free of rank $2$. The result follows from Lemma \ref{news2}. 

\medskip
Suppose then that $U= \alpha\beta$. If $\alpha\neq \alpha_1$ and $\beta\neq \beta_1$, then $\lbrace U,V\rbrace$ is Nielsen reduced, so $\langle U,V\rangle$ is free of rank $2$ and the result follows from Lemma \ref{news2}. Hence we may assume without loss of generality that $\alpha=\alpha_1$. If $\beta=\beta_1$, then $\langle U,V\rangle$ is isomorphic to $\mathbb{Z}$. In this case the result follows from Lemma \ref{news2}. 

\medskip
If $c_1=c_2=c$ then we can replace $w$ with $\hat{U}\beta\hat{U}\beta_1$, where $\hat{U}=c\alpha$. In this case $\<\hat{U}\>$ is free of rank $1$ and the result follows from Lemma \ref{news2}. 

Hence we may assume that $\beta\neq \beta_1$ and $c_1\neq c_2$, so $\langle \beta c_2,\beta_1c_1\rangle$ is free of rank $2$. Again we apply Lemma \ref{news2} to show that $G$ is non-trivial.
\end{proof}

\begin{prop}
If $U=( \alpha\beta)^{\pm 1}$ and $V=\alpha_1\beta_1\alpha_2$, then $G$ is non-trivial.
\end{prop}
\begin{proof}
Suppose $U=\alpha\beta$. We can assume that $\beta\neq \beta_1^{-1}$ as otherwise $G$ maps onto $B$, hence non-trivial. Also $\alpha_1\neq \alpha_2^{-1}$ by Lemma \ref{news1}. Since $\alpha_1\alpha_2\neq 1$, either $\alpha=\alpha_1$ and $\beta=\beta_1$ or  $\alpha= \alpha_2^{-1}$ and $\beta=\beta_1^{-1}$, as otherwise $\langle U,V\rangle$ is free. Since $\beta\ne\beta_1^{-1}$ we assume $\beta= \beta_1$ and $\alpha=\alpha_1$. If $c_1=c_2=c$, then $G$ surjects onto $(B*C)/N((\beta c)^2)$, so is non-trivial. Otherwise take $U'=c_2\alpha$ and $V'=\alpha_2c_1\alpha$, and so $\langle U',V'\rangle$ is free. Hence $G$ is non-trivial by Lemma \ref{news2}. 

The proof for the case where $U=( \alpha\beta)^{-1}$ is similar by symmetry.
\end{proof}

\begin{prop}
If $U=\alpha\beta\alpha_1$ and $V= \alpha_2\beta_1\alpha_3$, then $G$ is non-trivial.
\end{prop}

\begin{proof}
We may assume that $\beta\neq \beta_1^{-1}$, as otherwise $G$ surjects onto $B$. Without loss of generality, there are two possibilities to consider. Either $\alpha=\alpha_2$ and $\beta=\beta_1$ or $\alpha=\alpha_2$ and $\alpha_1=\alpha_3$. (Note that we can not have $\alpha=\alpha_3^{-1}$ and $\alpha_1=\alpha_2^{-1}$, for otherwise $w$ is contained in the normal closure of $B*C$).

\medskip
In the first case, we take take $U'=\alpha_1c_2\alpha$ and $V'=\alpha_3c_1\alpha$. By Lemma \ref{news2}, we can assume that $\<U',V'\>$ is not free. Hence either $c_1=c_2=c$ or $\alpha_1=\alpha_3$. In either case $G$ maps onto 

\begin{equation*}
\dfrac{(B*C)}{N((c\beta)^2} \quad \text{or} \quad \dfrac{(A*B)}{N((\alpha\beta \alpha_1)^2)}
\end{equation*}
respectively. Hence $G$ is non-trivial.

\medskip
In the second case where $\alpha=\alpha_2$ and $\alpha_1=\alpha_3$, we can replace $B$ with its conjugate by $\alpha$, and $w$ by $W=c_1\beta\tilde{\alpha}c_2\beta_1\tilde{\alpha}$, where $\tilde{\alpha}=\alpha\alpha_1$. Since $G$ is isomorphic to

\begin{equation*}
G'=\dfrac{(A*\alpha^{-1}B\alpha*C)}{N(W)},
\end{equation*}
the result follows from Lemma \ref{pp1}.
\end{proof}

\begin{prop}\label{pp4}
If $U=\alpha\beta\alpha_1$ and $V=\beta_1\alpha_2\beta_2$, then $G$ is non-trivial.
\end{prop}
\begin{proof}
In this case, the only possibility is $\alpha=\alpha_1^{-1}$ or $\beta_1=\beta_2^{-1}$. In either case, the result follows from Lemma \ref{news1}.
\end{proof}

Finally, we consider the case where $U=(\alpha\beta)^{\pm 1}$ and $V= \alpha_1\beta_1\alpha_2\beta_2$.
 In this case we can prove the stronger result that
each of $A*B,C$ embeds in $G$ (the Freiheitssatz).
To this end, standard arguments allow us to make the additional
assumption that each of $A,B,C$ is generated by the letters occurring in $w$.

 Since $\<U,V\>=C_p*K$, without loss of generality we have by Nielsen transformations that either 
\begin{enumerate}
\item $U=\alpha\beta$ and $V\in\lbrace\alpha\beta\alpha\beta_2, \alpha\beta\alpha_2\beta\rbrace$ with $\beta\neq \beta_2$, or
\item $U=\beta^{-1}\alpha^{-1}$ and $V\in\lbrace\alpha\beta\alpha\beta_2, \alpha\beta\alpha_2\beta\rbrace$ with $\alpha\neq\alpha_2$.
\end{enumerate}

\begin{rem}\label{Nel}
In (1) and (2) above we gave two forms of $V$. If $V=\alpha\beta\alpha_2\beta$ we can replace $U$ and $V$ by $\beta U\beta^{-1}$ and $\beta V\beta^{-1}$ respectively (or equivalently replace $C$ by $\beta^{-1}C\beta$) and interchange $A$ and $B$ to get the first form, $\alpha\beta\alpha\beta_2$.
\end{rem}

In what follows we regard $G$ as a one-relator product of $A*B$ and $C$. For convenience we let $U_1=\alpha\beta$ and $U_2=\beta^{-1}\alpha^{-1}$, so $U_2=U_1^{-1}$. We use  $R$ to denote a relator in  $G$ which is a cyclically reduced word  in $\lbrace U,V\rbrace$ and $\ell(R)$ denotes its length also as a word  in $\lbrace U,V\rbrace$ . 

\begin{defn}\label{index}
The \textit{index} of $R$ is the number of cyclic sub-words of the form $(UU)^{\pm 1}$, $(VV)^{\pm 1}$, $(VU^{-1})^{\pm 1}$ or $(U^{-1}V)^{\pm 1}$.
\end{defn}

Definition \ref{index} generalizes the notion of sign-index. Recall that the sign-index of $R$ is  $n$ (necessarily even) if a cyclic permutation of $R$ has the form $$W_1W_2^{-1}W_3\ldots W_{n-1}W_{n}^{-1},$$ with each $W_i$ a positive word  in $\lbrace U,V\rbrace$. In particular the index of $R$ is bounded below by its sign-index, and above by $\ell(R)$.

\medskip
By Remark \ref{Nel}

$$\lbrace U,V\rbrace =\lbrace (\alpha\beta)^{\pm 1}, \alpha\beta\alpha\beta_2\rbrace \xrightarrow{\text{Nielsen transformation}}\lbrace \alpha\beta, \beta^{-1}\beta_2\rbrace.$$

There are two possibilities to consider. If $\lbrace \alpha\beta, \beta^{-1}\beta_2\rbrace$ is Nielsen reduced, then $\<U,V\>$ is free (if $\beta^{-1}\beta_2$ has infinite order), or $\mathbb{Z}*\mathbb{Z}_m$ (if $\beta^{-1}\beta_2$ has order $m$). A second possibility is that $\lbrace \alpha\beta, \beta^{-1}\beta_2\rbrace$ is not Nielsen reduced. In which case $\beta$ is a power of $\beta^{-1}\beta_2$, so $B$ is cyclic (generated by $\beta^{-1}\beta_2$). In particular it follows that $\beta^{-1}\beta_2$ must have order at least $3$ (since by assumption $\beta\neq \beta_2$).

\begin{prop}\label{positive}
Suppose $R$ is a cyclically reduced word in $\lbrace U,V\rbrace$ of index $k$. Suppose also that $(\beta^{-1}\beta_2)^2\neq 1$. If $R$ is trivial, then $2k+\ell(R)\geq 12$. 
\end{prop}
\begin{proof}
Note that $\langle U\rangle \cap \langle V \rangle = 1$
so no word of the form $U^mV^n$ is a relator. In particular $\ell(R)\geq 4$. Hence we can assume that $k\leq 3$, and so in particular we only need to consider words with sign-index $0$ or $2$.

\medskip
First suppose that $R$ has sign-index $0$. Since $\beta\neq \beta_2$, $\lbrace U_1,V\rbrace$ generates a free sub-semigroup of $A*B$ of rank $2$. Hence we assume that $U=U_2$. 
Now $\ell(R)\le 12-3k$ where $k\le 3$ is the number of cyclic
subwords of the form $U^{\pm 2}$ or $V^{\pm 2}$.
This leaves us with a short list of words that can be checked directly to show that none is trivial.

\medskip
Suppose that $R$ has sign-index $2$. If $U=U_1$, we get an equality between two positive words in a free sub-semigroup. Since this can not happen, we assume that $U=U_2$.
We can assume that $\ell(R)=4$ or $5$ and $R$ has at most a single occurrence of the subword $U^{\pm2}$ or $V^{\pm2}$. Again there is a very short list of such words which  can checked directly to show that none is trivial.
\end{proof}

\begin{rem}\label{NR}
Note that if $(\beta^{-1}\beta_2)^2=1$, then $B$ can not be cyclic since that will imply that $\beta=\beta_2.$ In other words $\lbrace U, \beta^{-1}\beta_2\rbrace$ is  Nielsen reduced.
\end{rem}

\begin{lem}\label{news4}
If  $U=(\alpha\beta)^{\pm 1}$, $V=\alpha\beta\alpha\beta_2$ and $\lbrace U, \beta^{-1}\beta_2\rbrace$ is not Nielsen reduced, then 
each of $A*B,C$ embeds in $G$ via the natural map.
\end{lem}
\begin{proof}
Suppose by contradiction that the result fails, then we get a non-trivial $\mathcal{W}$-minimal spherical picture $M$ over $G$ 
where $\mathcal{W}$ is the set of non-trivial elements of $(A*B)\cup C$.

\medskip
Note that both $A=\langle\alpha\rangle$ and $B=\langle\beta^{-1}\beta_2\rangle$ are cyclic. So  by Theorem \ref{t33} we may assume that $C$ is not cyclic or dihedral.
Thus at most one of $c_1,c_2$ can have order $2$.  Without loss of generality we assume that $c_2$ has order greater than $2$.

\medskip
Assign angles to corners of a picture over $G$ as follows.  Every $c_1$-corner gets angle $0$, every $c_2$-corner gets angle $\pi/3$, and every $U$- and $V$-corner gets angle $5\pi/6$.  This ensures that vertices have curvature $0$, and $C$-regions have non-positive curvature.  However, $(A*B)$-regions can have positive curvature.  We overcome this by redistributing any such positive curvature to neighbouring negatively curved $C$-regions, as follows.

\medskip
Let $\Delta$  be  an interior $(A*B)$-region of $M$. We transfer $\pi/3$ of curvature across each of the edges of $\Delta$ joining a $c_1$-corner to $c_2$-corner of an adjacent $C$-region. 

\medskip
Now any interior $(A*B)$-region whose label is of the form $(U_2U_2V)^n$ has  label of index $n$ and curvature at most $\pi/2$. However,  $n>2$ by Remark \ref{NR}. Hence it has transferred at least $\pi$ of curvature to neighbours, so becomes negatively curved. Similarly, it follows from Proposition \ref{positive} that any interior $(A*B)$-region whose label is not of the form $(U_2U_2V)^n$ has  curvature at most $\pi$  and that it has transferred at least $\pi$ of curvature to neighbours, so  it becomes non-positively curved as well.

\medskip
A $C$-region $\Delta$ receives $\pi/3$ of positive curvature across each edge separating a $c_1$-corner from a $c_2$-corner. Suppose that in $\Delta$ there are $p~$  $c_1$-corners, $q~$ $c_2$-corners, and $r$ edges separating a $c_1$-corner from a $c_2$-corner. The  curvature of $\Delta$ after transfer is at most

\begin{equation*}
2\pi-(p+q)\pi + \dfrac{q\pi}{3} + \dfrac{r\pi}{3} \leq \dfrac{(6-2p-q)\pi}{3}.
\end{equation*}
If $2p+q\geq 6$, then $\Delta$ still has non-positive curvature after transfer. Suppose $\Delta$ is an interior $C$-region and $2p+q\leq 5$. 
Then either $p=0$ or $q=0$ (since $p,q\neq 1$). 
Hence also $r=0$: so there is no transfer of curvature into $\Delta$ and it remains non-positively curved.

\medskip
Since the curvature of the exceptional region is less than $4\pi$, we get a contradiction that curvature of $M$ is $4\pi$. Hence  each of $A*B,C$ embeds in $G$ via the natural map.
\end{proof}

\begin{lem}\label{news5}
If  $U=(\alpha\beta)^{\pm 1}$, $V=  \alpha\beta\alpha\beta_2$ and $\lbrace U, \beta^{-1}\beta_2\rbrace$ is Nielsen reduced, then 
each of $A*B,C$ embeds in $G$ via the natural map.
\end{lem}

\begin{figure}[h!]
\centering
\includegraphics[scale=0.25]{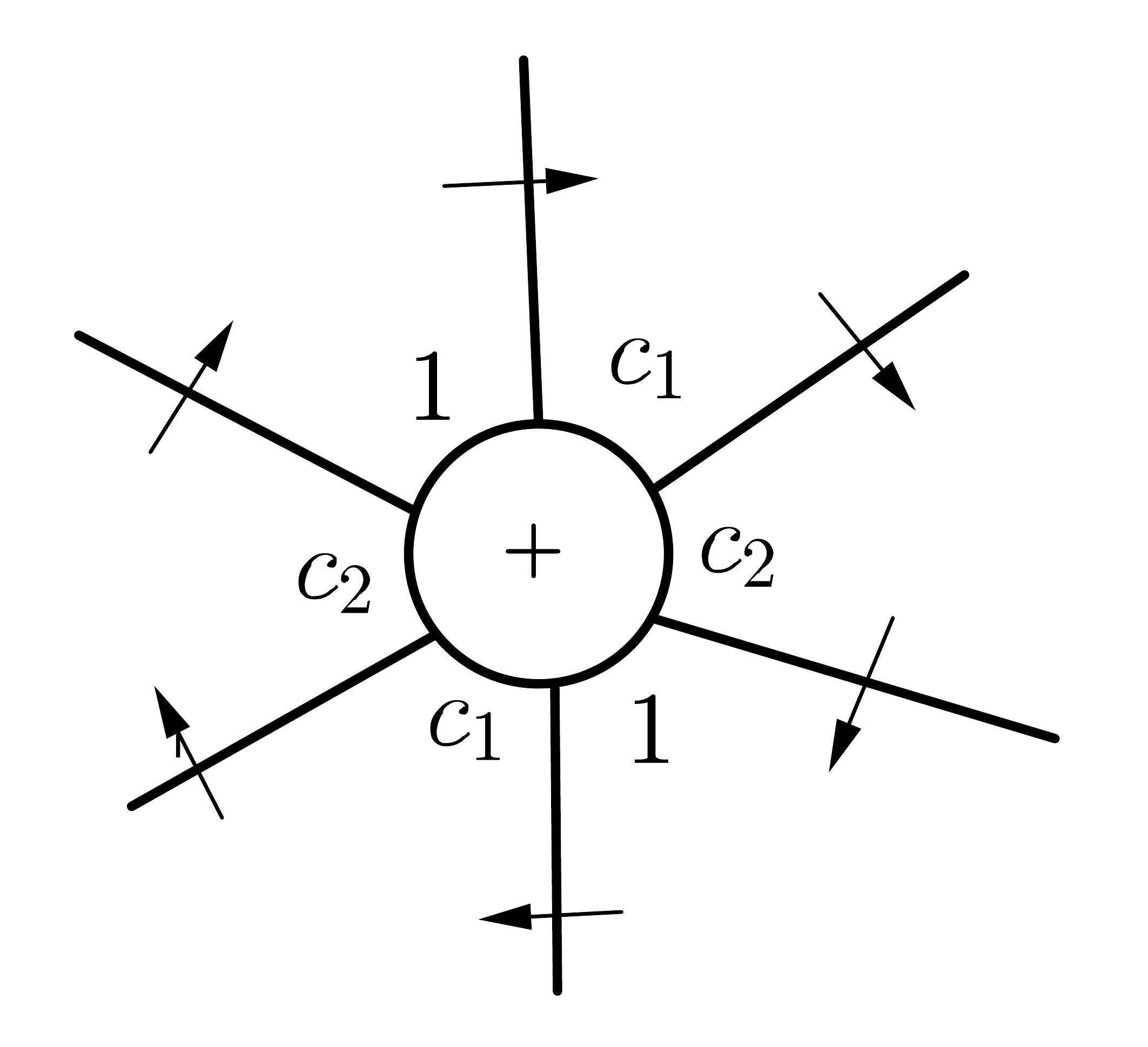}
\quad
\includegraphics[scale=0.25]{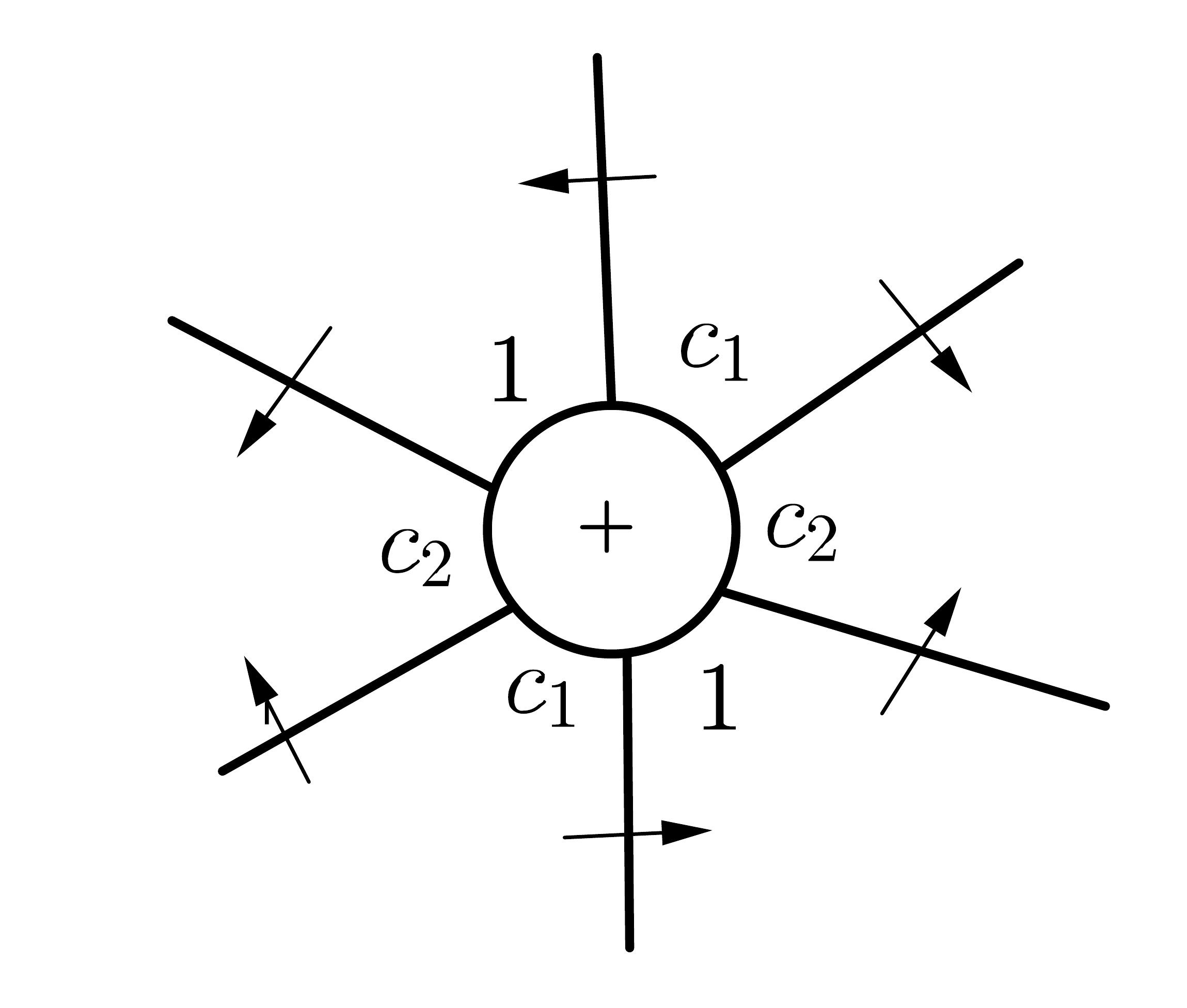}
\caption{Positively oriented vertices of $\Gamma$ when $n=2$. The figure on the left corresponds to a vertex of $\Gamma$ when $r=(t^ 2c_1tc_2)^n$ and the other is when $r=(t^{- 2}c_1tc_2)^n$.}
\label{figs}
\end{figure}

\begin{proof}
By assumption $\<U,V\>=\<U\>*\<\beta^{-1}\beta_2\>$, and is isomorphic to $\mathbb{Z}*\mathbb{Z}_n$, where $n>1$ is the order of $\beta^{-1}\beta_2$. Consider the relative presentation
$$H=\<C,t~|~r=(t^{\pm 2}c_1tc_2)^n\>.$$
The aim is to show that $G$ is the free product of $H$ and $A*B$ amalgamated over the subgroups $\<U,V\>$ and $\<t,c_1tc_2\>$.  To this end we must show that the latter is also isomorphic to $\mathbb{Z}*\mathbb{Z}_n$. In other words, any relation in $H$ which is a word in $\lbrace t,c_1tc_2\rbrace$ is a consequence of $(t^{\pm 2}c_1tc_2)^n$.

\medskip
If this is not so, we obtain a $\mathcal{W}$-minimal non-trivial picture $\Gamma$ over $H$ on $D^2$ where $\mathcal{W}$ is the set of non-trivial words in $\lbrace t,c_1tc_2\rbrace$. Figure \ref{figs} shows typical vertices with positive orientation in the case of $n=2$. Note that there are no $2$-zones with corners labelled by $1^{\pm 1}$, for in such case, either the two vertices cancel, or we can combine the vertex with the boundary. In both cases we get a smaller picture, thereby contradicting minimality of $\Gamma$. 

\medskip
Make regions of $\Gamma$ \textit{flat} by assigning angle $(d(\Delta)-2)\pi/d(\Delta)$ to each of the corners of a degree $d(\Delta)$  region $\Delta$. We claim that interior vertices $\Gamma$ are non-positively curved. The proof is in two stages depending on whether $r=(t^{2}c_1tc_2)^n$ or $r=(t^{-2}c_1tc_2)^n$.

\medskip
Suppose that $r=(t^{2}c_1tc_2)^n$. Since $r$ is a positive word, $\Gamma$ is bipartite. More precisely, only vertices of opposite orientations are adjacent in $\Gamma$. In particular this implies that regions have even degrees. Every interior vertex $v$ bounds at least two regions with a corner labelled $1$. By minimality of $\Gamma$, every such region has degree at least $4$.
Also every $2$-zone gives the relation $c_1=c_2$, and so each of the two regions on both sides of the $2$-zone has  a corner labelled $1$, hence is at least a $4$-gon. Hence $v$ bounds at least four regions of degree at least $4$ and so is non-positively curved.

\medskip
The case of $(t^{-2}c_1tc_2)^n$ is slightly different as regions can  have odd degree. Note that any corner is either a \textit{source} (the two arrows point outwards), \textit{sink} (the two arrows point inwards),  or \textit{saddle} (one arrow points inwards and the other points outwards) depending on whether it is a $c_1$-, $c_2$- or $1$-corner (see Figure \ref{figs}). So in particular any $2$-zone gives the relation $c_1^2=1$ or $c_2^2=1$. If an interior vertex $v$ does not bound a $2$-zone, then $v$ satisfies $C(3n)$. Suppose it does.
 As before any region adjacent to a $2$-zone has a $1$-corner
  (a saddle). But any region must have an even number of saddles, no two of which are adjacent (for otherwise a cancellation would be possible). It follows that such a region has degree at least $4$. There are at least  two regions with $1$-corner at $v$. If $v$ bounds only one $2$-zone, then it has degree at least $5$ and bounds at least three regions of degree $4$. Otherwise $v$ bounds at least four $4$-gons. In all cases $v$ has non-positive curvature.

\medskip
It follows  that there exists a boundary vertex of degree at most $3$. This is clearly impossible if $n>2$ (since we will get a $2$-zone with corners labelled $1$). So we assume that $n=2$.
An argument similar to the ones given above shows that such a vertex must connect to $\partial D^2$ by an $\omega$-zone, with $\omega\ge 3$. It follows that either one of $c_1$ or $c_2$ is trivial or we can combine such a vertex with  $\partial D^2$ to get a smaller picture. Both possibilities lead to a contradiction which completes the proof.
\end{proof}

\begin{proof}[Proof of Theorem \ref{t2}]
By earlier comments we can assume $6\leq \ell(w)<9$ and $w$ has the form $w=c_{1}Uc_{2}V$ (up to cyclic permutation) where $U,V \in A*B$  and $ c_1,c_2\in C$ with $c_1c_2\neq 1$. It follows from  Grushko's theorem that the subgroup of $A*B$ generated by $U$ and $V$ is isomorphic to one of the following:

\medskip
\begin{enumerate}
\item Conjugate to subgroup of $A$ (or $B$).
\item Free group of rank one.
\item Free group of rank two.
\item Free product of two finite  cyclic groups.
\item Free product of finite and infinite cyclic groups.
\end{enumerate}
In the case of part (1) the result  is immediate. 
Parts (2) and (3) follow from Lemma \ref{news2}. And finally parts (4) and (5) follow from
Lemmas \ref{news3}, \ref{news4} and \ref{news5}.
\end{proof}

\bibliographystyle{acm}

\end{document}